\documentclass[11pt,reqno]{amsart}
\usepackage[T1]{fontenc}
\usepackage{lmodern,amsmath,amssymb,amsthm,mathtools,microtype}
\usepackage[margin=1in]{geometry}
\usepackage{xcolor}
\usepackage[colorlinks=true,linkcolor=blue!50!black,citecolor=blue!50!black,urlcolor=blue!50!black]{hyperref}
\usepackage{zref-clever}
\zcsetup{cap=true,abbrev=false}
\usepackage{aliascnt}
\newtheorem{theorem}{Theorem}[section]
\newaliascnt{lemma}{theorem}
\newtheorem{lemma}[lemma]{Lemma}
\aliascntresetthe{lemma}
\newaliascnt{corollary}{theorem}
\newtheorem{corollary}[corollary]{Corollary}
\aliascntresetthe{corollary}
\newaliascnt{remark}{theorem}
\theoremstyle{remark}
\newtheorem{remark}[remark]{Remark}
\aliascntresetthe{remark}
\theoremstyle{plain}
\newcommand{\E}{\mathbb E}
\newcommand{\Pp}{\mathbb P}
\newcommand{\R}{\mathbb R}
\newcommand{\cR}{\mathcal R}
\DeclareMathOperator{\Cov}{Cov}
\DeclareMathOperator{\Var}{Var}
\DeclareMathOperator{\Ber}{Bernoulli}
\newcommand{\one}{\mathbf 1}
\newcommand{\norminf}[1]{\lVert #1\rVert_\infty}
\newcommand{\dd}{\,\mathrm du}
\allowdisplaybreaks[1]
\title[Optimal central limit theorem in high dimensions]{Optimal central limit theorem for bounded random variables in high dimensions}
\author{P.\ M.\ Aronow}
\author{Patrick Lopatto}
\date{\today}

\begin{document}
\begin{abstract}
Let $W=n^{-1/2}\sum_{i=1}^n X_i$, where the $X_i$ are independent centered random vectors in
$\R^p$ with $|X_{ij}|\le B$ almost surely. Suppose that $\Cov(W)$ has unit diagonal and smallest
eigenvalue at least $b^2>0$.
We prove that the distance between $W$ and a Gaussian vector with the same covariance, uniformly
over axis-aligned rectangles, is at most $C\min\{1,b^{-2}Bn^{-1/2}\log^{3/2}(ep)\}$. 
For fixed $b$, the dependence on summand size and dimension matches known lower bounds in
growing-dimensional regimes.
The proof combines a concentration estimate near rectangle boundaries with a carefully chosen
Gaussian comparison. 
\end{abstract}
\maketitle

\section{Introduction}\label{sec:introduction}

Many high-dimensional statistical procedures depend on the joint behavior of a large number of coordinates.
For example, simultaneous confidence regions, family-wise error control in multiple testing, and the
calibration of penalty parameters all require control of probabilities of events defined by many coordinatewise
inequalities. A natural approach is to replace the underlying sum by a Gaussian vector with the same
covariance and to control the approximation uniformly over such events
\cite{CCK2013,CCK2017,CCK2023}.
We focus on axis-aligned rectangles, a class broad enough to cover the applications above
but structured enough for accurate Gaussian approximation even when the dimension $p$ is much larger
than the sample size $n$. The usefulness of the central limit theorem for this class therefore depends on its
quantitative behavior as the dimension grows. We consider here sums of bounded independent random vectors, and we are interested in how the approximation error depends on the dimension and on the size of the summands.

Recent work has obtained increasingly general high-dimensional Gaussian
approximation results and, under additional structural assumptions, nearly
$n^{-1/2}$ rates. Chernozhukov, Chetverikov and Kato \cite{CCK2013} established Gaussian and
multiplier-bootstrap approximations for maxima of sums of high-dimensional random vectors, and their
later work \cite{CCK2017} treated approximation uniformly over rectangles. Subsequent results obtained
nearly parametric dependence on the sample size under stronger or different structural assumptions.
Fang and Koike \cite{FK2021} considered independent identically distributed log-concave summands,
Lopes \cite{Lopes2022} treated independent identically distributed vectors with sub-Gaussian or
sub-exponential coordinates, and Kuchibhotla and Rinaldo \cite{KR2020} allowed non-identically
distributed summands but imposed nondegeneracy conditions on the individual summand covariances. Fang and Koike subsequently obtained rectangle CLTs for i.i.d.\
log-concave summands with arbitrary, possibly singular, covariance matrices
\cite{FK2024}. 
These results are not directly comparable through their dependence on $n$ alone, as the dependence on
the dimension and the form of the covariance assumptions are also essential.

For the bounded setting considered here, the closest result is due to Chernozhukov, Chetverikov and
Koike \cite[Theorem 2.1 and Corollary 2.1]{CCK2023}. For a normalized sum
$n^{-1/2}\sum_i X_i$ of independent centered vectors with $|X_{ij}|\le B$, they obtained a rate of 
\[
 C\frac{B\log^{3/2}(ep)\log n}{b^2\sqrt n},
\]
where $C$ is an absolute constant, provided that $n,p\ge3$ and the covariance of the sum has unit diagonal
and is bounded below by $b^2I_p$. Their Proposition 2.1 also gives lower bounds of order
$B\log^{3/2}(p)/\sqrt n$ in growing-dimensional regimes. Thus, for bounded summands and under
aggregate covariance nondegeneracy, the dependence on the dimension and summand size was already
sharp, while the upper bound retained an additional factor of $\log n$.

The same authors removed that factor for sums convolved with an independent, uniformly nondegenerate
Gaussian \cite[Theorem 4.1 and Remark 4.1]{CCK2023}.
This smoothing assumption excludes general bounded sums, which may have atoms.
We remove the factor $\log n$ without Gaussian smoothing, under the same condition on the
aggregate covariance and with the same dependence $b^{-2}$ as in \cite[Corollary 2.1]{CCK2023}.

\subsection{Main result}

Let $\cR_p$ be the class of rectangles $\prod_{j=1}^p(\ell_j,u_j]$,
where $-\infty\le\ell_j<u_j\le\infty$. For random vectors $U,V$ in
$\R^p$, define
\[
 \rho_{\cR}(U,V)
 =\sup_{A\in\cR_p}|\Pp(U\in A)-\Pp(V\in A)|.
\]

\begin{theorem}\label{thm:main}
Let $n,p$ be positive integers, $b\in(0,1]$, and $B>0$.
Let $X_1,\ldots,X_n$ be independent centered random vectors in $\R^p$ with
$|X_{ij}|\le B$ almost surely, and set
\[
 W=\frac1{\sqrt n}\sum_{i=1}^nX_i,\qquad
 \Sigma=\Cov(W),\qquad Z\sim N(0,\Sigma).
\]
If $\Sigma_{jj}=1$ for every $j$ and $\Sigma\succeq b^2I_p$, then
\begin{equation}\label{eq:main}
 \rho_{\cR}(W,Z)\le C\min\left\{1,\frac{B\log^{3/2}(ep)}{b^2\sqrt n}\right\},
\end{equation}
where $C$ is an absolute constant.
\end{theorem}

The theorem includes summands with rank-one covariance matrices and atomic laws, and requires neither identically distributed
summands nor independent coordinates.
For fixed $B$ and $b$, the approximation error tends to zero whenever $\log p=o(n^{1/3})$.

For fixed $b$, the order in \eqref{eq:main} is optimal in growing-dimensional
regimes. For example, the conditions of \cite[Proposition 2.1]{CCK2023} hold for fixed $B=2$ and
$p_n=\lceil\exp(n^{1/4})\rceil$. Their construction has independent coordinates and hence aggregate
covariance $I_{p_n}$, and gives a lower bound of order $B\log^{3/2}(p_n)/\sqrt n$.

\subsection{Ideas of the proof}

For the proof, set $\xi_i=X_i/\sqrt n$ and $\delta=B/\sqrt n$, so that
$W=\sum_i\xi_i$ and $\norminf{\xi_i}\le\delta$, and write $L=\log(ep)$.
The proof is based on an interpolation between a smoothed version of the original sum and a Gaussian
vector with matching covariance. 
We interpolate so that, as the interpolation parameter $\lambda$ increases, the covariance of the
summands decreases while the Gaussian covariance increases by the same amount, keeping the total
covariance fixed. The Gaussian component has covariance $(\lambda+t)\Sigma$ at interpolation time
$\lambda$, where $t\Sigma$ is the amount of Gaussian smoothing present at the initial endpoint.

If each summand is retained independently with probability $1-\lambda$ and is
not rescaled, both its covariance and its third-moment tensor are multiplied
by $1-\lambda$. The cubic term in the differentiated interpolation then has an
averaged bound proportional to $(\lambda+t)^{-1}$, whose integral is
$\log(1+1/t)$. As explained below, the derivative estimates require
\[
 t\asymp \frac{\delta^2L}{b^2}
 =\frac{B^2\log(ep)}{b^2n},
\]
so $\log(1/t)$ contains exactly the kind of sample-size logarithm that we want to avoid.

We instead choose the interpolation so that the covariance changes linearly in $\lambda$, while the third
moments change only quadratically. Their derivative is therefore of order $\lambda$, and the cubic
term is bounded by a multiple of $\lambda / (\lambda + t)$, which is integrable uniformly in $t$.

The first technical ingredient is a concentration estimate near rectangle boundaries. 
For a sum of independent centered vectors with unit coordinate variances and summands bounded by
$\eta$, the probability of lying within distance $a$ of a rectangle boundary is at most
$C(a+\eta)\sqrt L$. We prove this by independently replacing each summand by a fresh copy with a small probability
and comparing Kane's noise-sensitivity bound \cite[Corollary 1]{Kane2014} with a lower bound on the
conditional probability that this perturbation moves a point near the rectangle boundary across one
of its faces;  see \zcref{sec:concentration}.

The second ingredient is a local Gaussian derivative estimate. Related estimates for
Gaussian-smoothed rectangle indicators appear in Fang and Koike
\cite[Lemmas 2.2 and 2.3]{FK2021} and in Chernozhukov, Chetverikov and Koike
\cite[Section 6]{CCK2023}. In the formulation used here, one takes a local supremum of each derivative entry and then sums over the entries; the resulting quantity is bounded by a certain boundary-layer probability, which can then be averaged using the preceding concentration estimate. 
At Gaussian covariance $s\Sigma$, the relevant boundary layer has width of order
$b\sqrt{s/L}$. The Taylor remainders evaluate Gaussian derivatives at arguments shifted by at most
a constant multiple of $\delta$, so this boundary layer must be at least of that size. This explains
the choice $t\asymp\delta^2L/b^2$ above.
Throughout the interpolation, either the bounded sum or the Gaussian component has coordinate
variances at least $1/2$; conditioning on the other therefore makes the boundary-concentration
estimate available along the entire path. See \zcref{sec:envelope}.

For $0\le\lambda\le1$, let
\[
 V_\lambda=\sum_i(1+\lambda)B_{i,\lambda}\xi_i,
 \qquad
 B_{i,\lambda}\sim\Ber\left(\frac{1-\lambda}{(1+\lambda)^2}\right),
\]
where the Bernoulli variables are mutually independent and independent of the summands.
Each summand's covariance is multiplied by $1-\lambda$, while its third-moment tensor is multiplied
by $1-\lambda^2$. Adding an independent Gaussian with covariance $(\lambda+t)\Sigma$ therefore keeps
the total covariance fixed at $(1+t)\Sigma$. Write
$P_A(\lambda)=\Pp(V_\lambda+G_{(\lambda+t)\Sigma}\in A)$, where $G_\Gamma\sim N(0,\Gamma)$ is independent of
the summands. Differentiating in $\lambda$, the leading covariance terms cancel, leaving
fourth-order remainders. The averaged derivative bounds then give
\[
 |P_A'(\lambda)|\le C\left[
 b^{-2}\delta L^{3/2}\frac{\lambda}{\lambda+t}
 +b^{-3}\delta^2L^2(\lambda+t)^{-3/2}\right].
\]
The cubic term integrates without a logarithm because $\lambda/(\lambda+t)\le1$, while the
fourth-order remainder contributes at most $Cb^{-3}\delta^2L^2/\sqrt t$. Hence
\[
 |P_A(1)-P_A(0)|
 \le Cb^{-2}\delta L^{3/2}
    +Cb^{-3}\frac{\delta^2L^2}{\sqrt t}
 \le Cb^{-2}\delta L^{3/2}.
\]
Since $P_A$ connects $W+G_{t\Sigma}$ to $G_{(1+t)\Sigma}$, it remains only to remove the smoothing.
The same boundary-concentration estimates bound this error by
$C(\sqrt t\,L+\delta\sqrt L)$, which is no larger than the claimed bound for the choice of $t$ above.  This completes the comparison; the details are in
\zcref{sec:comparison}.

\subsection*{Acknowledgments}

P.\ L.\ was partially supported by NSF grant DMS-2450004. This paper was written
by the authors with the assistance of large language models, which were used to
suggest arguments, contribute to drafting and revision, and write code
for computational checks.

\section{Rectangle concentration by resampling}\label{sec:concentration}

Throughout, $L=\log(ep)$.
Constants denoted by $C$ are absolute and may change between occurrences; subscripts indicate
dependence only on the displayed parameters.
For a rectangle $A=\prod_{j=1}^p(\ell_j,u_j]$ and a real number $a$, set
\[
 A^a=\prod_{j=1}^p(\ell_j-a,u_j+a].
\]
Infinite endpoints remain infinite. If this operation makes any coordinate interval empty, we take the resulting rectangle to be empty.  Write $\phi$
and $\Phi$ for the standard normal density and distribution function, and $\bar\Phi=1-\Phi$.

\begin{lemma}\label{lem:concentration}
Let $S=\sum_{i=1}^n\zeta_i$, where the $\zeta_i$ are independent centered random vectors in $\R^p$
satisfying $\norminf{\zeta_i}\le\eta$ almost surely for every $i$, and suppose
$\Var(S_j)=1$ for every $j=1,\ldots,p$. For every $A\in\cR_p$ and $a>0$,
\begin{equation}\label{eq:shell}
 \Pp(S\in A^a\setminus A^{-a})\le C\min\{1,(a+\eta)\sqrt L\}.
\end{equation}
The same bound holds after adding any independent random vector to $S$. If instead all coordinate
variances equal $v>0$, we may replace $(a+\eta)\sqrt L$ by $(a+\eta)\sqrt{L/v}$.
\end{lemma}

\begin{proof}
We first represent summand resampling as Boolean sign flips, and then show that resampling has a
uniformly positive probability of pushing a point near the rectangle boundary across a face.

Kane \cite[Corollary 1]{Kane2014} proves that, for the indicator $f$ of an intersection of $k$
affine halfspaces on a uniform Boolean cube, independent sign flips with probability
$\theta\in(0,1/2]$ give
\begin{equation}\label{eq:kane}
 \Pp\{f(\varepsilon)\ne f(\varepsilon^\theta)\}
 \le C\sqrt{\theta\log(ek)}.
\end{equation}

Let $S^\theta$ denote the sum obtained from $S$ by independently replacing each $\zeta_i$
by a fresh independent copy with probability $\theta$.
For each $i$, let $\zeta_i^+$ and $\zeta_i^-$ be independent copies of $\zeta_i$, and let
$\varepsilon_i$ be an independent uniform sign. Set
\[
 \widetilde S(\varepsilon)
 =\sum_i\left(
 \frac{1+\varepsilon_i}{2}\zeta_i^+
 +\frac{1-\varepsilon_i}{2}\zeta_i^-
 \right).
\]
Then $\widetilde S(\varepsilon)\stackrel{d}{=}S$. If $\varepsilon^\theta$ is obtained by
independently flipping each sign with probability $\theta$, a flip simply selects the other
independent copy, and hence
\[
 \bigl(\widetilde S(\varepsilon),\widetilde S(\varepsilon^\theta)\bigr)
 \stackrel{d}{=}(S,S^\theta).
\]

Conditional on the candidate pairs, membership in a rectangle is an intersection of at most $2p$
affine halfspace conditions in the signs. Since the Boolean cube is finite, any strict threshold can
be perturbed without changing the resulting Boolean function, so the choice of open or closed
halfspaces is immaterial. Applying
\eqref{eq:kane} conditionally and then averaging gives, for every rectangle $Q$,
\begin{equation}\label{eq:row-noise}
 \Pp\{\one_Q(S)\ne\one_Q(S^\theta)\}\le C\sqrt{\theta L}.
\end{equation}

For the remainder of the proof, use the equivalent resampling representation
\[
 S^\theta=\sum_i\bigl[(1-B_i)\zeta_i+B_i\zeta_i'\bigr],
\]
where the $B_i$ are independent $\Ber(\theta)$ variables and the $\zeta_i'$ are independent
copies of the $\zeta_i$, independent also of the $B_i$. Let
$\mathcal F=\sigma(\zeta_1,\ldots,\zeta_n)$. Set
$D=S^\theta-S$ and, for each coordinate $j$,
\[
 Y_{ij}=B_i(\zeta'_{ij}-\zeta_{ij})+\theta\zeta_{ij}.
\]
For every fixed $j$, the variables $Y_{1j},\ldots,Y_{nj}$ are
conditionally independent given $\mathcal F$ and satisfy
$\E(Y_{ij}\mid\mathcal F)=0$, and 
\[
 D_j+\theta S_j=\sum_iY_{ij}.
\]
Moreover, $|Y_{ij}|\le2\eta$ almost surely, and
\[
 \Var(Y_{ij}\mid\mathcal F)
 =\theta\E\zeta_{ij}^2+\theta(1-\theta)\zeta_{ij}^2.
\]
Hence
\begin{equation}\label{eq:conditional-var}
 v_j:=\Var(D_j\mid\mathcal F)
 =\sum_i\Var(Y_{ij}\mid\mathcal F)
 =\theta+\theta(1-\theta)\sum_i\zeta_{ij}^2
 \ge\theta,
\end{equation}
where we used $\sum_i\E\zeta_{ij}^2=1$.

Since each $Y_{ij}$ is centered and bounded by $2\eta$,
\[
 \sum_i\E(|Y_{ij}|^3\mid\mathcal F)
 \le2\eta\sum_i\E(Y_{ij}^2\mid\mathcal F)
 =2\eta v_j.
\]
The scalar Berry--Esseen theorem therefore gives, almost surely with respect to $\mathcal F$,
\begin{equation}\label{eq:conditional-be}
 \sup_z\left|
 \Pp\left(
 \frac{D_j+\theta S_j}{\sqrt{v_j}}\le z
 \,\middle|\,\mathcal F
 \right)-\Phi(z)
 \right|
 \le C\eta/\sqrt\theta.
\end{equation}
By continuity of $\Phi$, the same estimate holds with strict inequalities.

For a rectangle $Q$, let $\partial_d Q$ denote the set of points in $Q$ within distance $d$ of
at least one finite face. Set
\[
 q=d+\eta,\qquad \theta=Kq^2,
\]
where $K$ is a large absolute constant, to be fixed below. We may assume that $q\sqrt L\le c$, for a
sufficiently small absolute constant $c$, since otherwise the desired bound is trivial. With $c$
small enough, this also ensures $\theta\le1/2$.

Suppose that $S\in\partial_d Q$ and $\norminf S\le1/(4\sqrt\theta)$. Choose the first finite face
within distance $d$ in an arbitrary fixed ordering, and let $j$ be its coordinate. Conditional on $\mathcal F$,
crossing this face requires a displacement of at most $d$ in the outward direction. Suppose first that the selected face is the upper face in coordinate $j$, with endpoint $u_j$.
Since $S$ lies within distance $d$ of this face,
\[
 0\le u_j-S_j\le d.
\]
An outward crossing occurs if $D_j>u_j-S_j$. Because the Berry--Esseen estimate
\eqref{eq:conditional-be} is centered at $\E(D_j\mid\mathcal F)=-\theta S_j$, this condition is
\[
 \frac{D_j+\theta S_j}{\sqrt{v_j}}
 >
 \frac{u_j-S_j+\theta S_j}{\sqrt{v_j}}.
\]
On $\norminf S\le1/(4\sqrt\theta)$, the threshold on the right satisfies
\[
 \frac{u_j-S_j+\theta S_j}{\sqrt{v_j}}
 \le \frac{d+\theta|S_j|}{\sqrt{v_j}}
 \le \frac d{\sqrt\theta}+\sqrt\theta|S_j|
 \le K^{-1/2}+\frac14,
\]
where we used $v_j\ge\theta$, $d\le q$, and $\theta=Kq^2$.
For a lower face, the same argument applies after reversing signs. Thus
\eqref{eq:conditional-be} shows that the conditional probability of crossing outward is at least
\[
 1-\Phi(K^{-1/2}+1/4)-C\eta/\sqrt\theta.
\] 
Since $\eta/\sqrt\theta\le K^{-1/2}$, we may choose $K$ so that this probability is bounded below
by an absolute constant $c_0>0$.

Crossing the selected face forces $S^\theta\notin Q$. Hence, by
\eqref{eq:row-noise},
\[
 c_0\,\Pp\left(
 S\in\partial_d Q,\,
 \norminf S\le\frac1{4\sqrt\theta}
 \right)
 \le
 \Pp\{\one_Q(S)\ne\one_Q(S^\theta)\}
 \le C\sqrt{\theta L}
 =Cq\sqrt L.
\]
Therefore
\[
 \Pp(S\in\partial_d Q)
 \le Cq\sqrt L+
 \Pp\left\{\norminf S>\frac1{4\sqrt\theta}\right\}.
\]

It remains to bound the second term. Bernstein's inequality and a union bound give
\[
 \Pp\left\{\norminf S>\frac1{4\sqrt\theta}\right\}
 \le 2\exp(L-c_1/\theta),
\]
for an absolute constant $c_1>0$; here we use $\eta/\sqrt\theta\le K^{-1/2}$.
Since $\theta=Kq^2$ and $q\sqrt L\le c$, decreasing $c$ if necessary gives
$L\le c_1/(2\theta)$. Therefore
\[
 \Pp\left\{\norminf S>\frac1{4\sqrt\theta}\right\}
 \le 2e^{-c_1/(2\theta)}
 \le Cq\sqrt L,
\]
and consequently
\begin{equation}\label{eq:boundary-layer}
 \Pp(S\in\partial_d Q)\le C(d+\eta)\sqrt L.
\end{equation}

We now apply this with $Q=A^a$ and $d=2a$. If $A^{-a}$ is nonempty, every point of
$A^a\setminus A^{-a}$ lies within distance $2a$ of a finite face of $A^a$. If $A^{-a}$ is empty,
then for some coordinate $j$ the interval $(\ell_j+a,u_j-a]$ is empty, so
$u_j-\ell_j\le2a$. The corresponding coordinate interval of $A^a$ then has length at most $4a$,
and every point of $A^a$ lies within distance $2a$ of one of its two faces. Thus in all cases
\[
 A^a\setminus A^{-a}\subseteq\partial_{2a}(A^a).
\]
Equation \eqref{eq:boundary-layer} proves \eqref{eq:shell}.

Finally, adding an independent random vector only translates the rectangle after conditioning on
that vector, so the same estimate applies. If all coordinate variances equal $v>0$, apply the unit
variance result to $S/\sqrt v$; the boundary width and summand bound become $a/\sqrt v$ and
$\eta/\sqrt v$, respectively. This proves the remaining assertions.
\end{proof}

The following Gaussian boundary-layer estimate is a rectangle form of Nazarov's inequality
\cite{Nazarov2003}; see also \cite[Lemma A.1]{CCK2017}.
To keep the argument self-contained, we recover it here from \zcref{lem:concentration} by approximation, including singular covariance matrices.

\begin{corollary}\label{cor:gauss-shell}
If $G\sim N(0,\Sigma)$ and $\Sigma_{jj}=1$, including possibly singular $\Sigma$, then, for every
$A\in\cR_p$ and $a>0$,
\begin{equation}\label{eq:gauss-shell}
 \Pp(G\in A^a\setminus A^{-a})\le Ca\sqrt L.
\end{equation}
\end{corollary}
\begin{proof}
Choose $v_1,\ldots,v_p\in\R^p$ such that
\[
 \Sigma=\sum_{k=1}^p v_kv_k^T,
\]
and let $(\varepsilon_{kl})$ be independent uniform signs. Set
\[
 G_m=\frac1{\sqrt m}\sum_{k=1}^p\sum_{l=1}^m v_k\varepsilon_{kl}.
\]
Then $\Cov(G_m)=\Sigma$, while each summand has sup norm at most
\[
 \eta_m=\frac1{\sqrt m}\max_{1\le k\le p}\norminf{v_k}\longrightarrow0.
\]
Since $\Sigma_{jj}=1$, \zcref{lem:concentration} gives
\[
 \Pp(G_m\in A^a\setminus A^{-a})
 \le C(a+\eta_m)\sqrt L.
\]

By the finite-dimensional central limit theorem, $G_m$ converges in distribution to
$G\sim N(0,\Sigma)$. Every marginal of $G$ is standard normal, so the boundary of
$A^a\setminus A^{-a}$ has $G$-probability zero. Hence
\[
 \Pp(G_m\in A^a\setminus A^{-a})
 \longrightarrow
 \Pp(G\in A^a\setminus A^{-a}).
\]
Letting $m\to\infty$ proves \eqref{eq:gauss-shell}.
\end{proof}
\begin{remark}
For comparison, O'Donnell, Servedio and Tan
\cite[Theorem 7.1]{OST2022} prove an orthant-boundary anticoncentration bound for
linear images of independent uniform signs, under regularity assumptions on the
coefficients of the defining linear forms. In that setting, their result provides
an alternative route to a boundary estimate of this type. In contrast, \zcref{lem:concentration} allows arbitrary bounded summand laws,
without assuming a representation as a linear image of independent signs.

Bong, Kuchibhotla and Rinaldo \cite{BKR2026} obtain high-dimensional CLTs for
$m$-dependent data by coupling anticoncentration and Berry--Esseen estimates
through induction. In \zcref{lem:concentration}, the boundary-concentration
estimate instead follows directly from resampling, without an inductive Gaussian
approximation bound.
\end{remark}
\section{Local bounds for Gaussian derivatives}\label{sec:envelope}

For a positive definite covariance matrix $\Gamma$, define
\[
 g_{A,\Gamma}(x)=\Pp(x+G_\Gamma\in A),\qquad G_\Gamma\sim N(0,\Gamma).
\]
Whenever an auxiliary Gaussian vector is added to another random vector,
the two are taken independent unless stated otherwise.

\begin{lemma}\label{lem:envelope}
Let $\Gamma\succeq\sigma^2I_p$ with $\sigma>0$, and set
$h=\sigma/(3\sqrt L)$. For every positive integer $m$, every $A\in\cR_p$, and every
$x\in\R^p$,
\begin{equation}\label{eq:envelope}
 \sum_{j_1,\ldots,j_m=1}^p
 \sup_{\norminf z\le h}
 |\partial_{j_1\cdots j_m}g_{A,\Gamma}(x+z)|
 \le
 C_m h^{-m}
 \bigl[g_{A^{4h},\Gamma}(x)-g_{A^{-2h},\Gamma}(x)\bigr].
\end{equation}
\end{lemma}
We first derive the averaged estimate used in \zcref{sec:comparison},
and then prove \zcref{lem:envelope}.

\begin{corollary}\label{cor:averaged}
Let $b\in(0,1]$ and let $\Sigma\succeq b^2I_p$ satisfy $\Sigma_{jj}=1$ for
$j=1,\ldots,p$. Let $N\ge1$ and
\[
 S=\sum_{i=1}^N\zeta_i,
\]
where the $\zeta_i$ are independent centered random vectors in $\R^p$ satisfying
$\norminf{\zeta_i}\le\eta$ almost surely for every $i$. Suppose
$\Var(S_j)=v$ for $j=1,\ldots,p$, where $v\ge0$.
For $s>0$, set
\[
 h=\frac{b\sqrt s}{3\sqrt L}.
\]
If $\max\{v,s\}\ge1/2$ and $h\ge2\eta$, then, for every positive integer $m$ and every
$A\in\cR_p$,
\begin{equation}\label{eq:averaged}
 \E\sum_{j_1,\ldots,j_m=1}^p
 \sup_{\norminf z\le2\eta}
 |\partial_{j_1\cdots j_m}g_{A,s\Sigma}(S+z)|
 \le
 C_m b^{1-m}L^{m/2}s^{-(m-1)/2}.
\end{equation}
\end{corollary}

\begin{proof}
Since $2\eta\le h$, \zcref{lem:envelope} gives
\[
\E\sum_{j_1,\ldots,j_m=1}^p
\sup_{\norminf z\le2\eta}
|\partial_{j_1\cdots j_m}g_{A,s\Sigma}(S+z)|
\le
C_m h^{-m}
\Pp\bigl(S+G_{s\Sigma}\in A^{4h}\setminus A^{-2h}\bigr).
\]
Enlarging the boundary layer, it remains to bound
\[
\Pp\bigl(S+G_{s\Sigma}\in A^{4h}\setminus A^{-4h}\bigr).
\]

If $v\ge1/2$, condition on $G_{s\Sigma}$ and apply
\zcref{lem:concentration} to $S$. Since the summands are bounded by
$\eta$ and $h\ge2\eta$, the probability above is at most
\[
C(h+\eta)\sqrt{L/v}\le Ch\sqrt L.
\]

If $v<1/2$, then $s\ge1/2$. Conditioning instead on $S$, rescaling by $s^{-1/2}$, and applying
\zcref{cor:gauss-shell} to $G_{s\Sigma}$ gives
\[
 Ch\sqrt{L/s}\le Ch\sqrt L.
\]
Thus the left side of \eqref{eq:averaged} is at most
\[
 C_m h^{1-m}\sqrt L
 \le C_m b^{1-m}L^{m/2}s^{-(m-1)/2},
\]
which proves the claim.
\end{proof}

\begin{proof}[Proof of \zcref{lem:envelope}]
The proof has four steps.

\medskip
\emph{Step 1: One-dimensional derivative bounds.}
The left side is zero if the rectangle is empty, so we may assume that all coordinate intervals are
nonempty. First take $\Gamma=I_p$, so $h=1/(3\sqrt L)$. Absorbing the translation by $x$ into the
interval endpoints reduces the problem to the following one-dimensional estimate. For an interval
$I\subseteq\R$, let
\[
 q_I(y)=\Pp(G+y\in I),
\]
where $G$ is standard normal. We claim that, for $1\le k\le m$,
\begin{equation}\label{eq:scalar-strip}
 \sup_{|y|\le h}|q_I^{(k)}(y)|
 \le C_kh^{-k}\bigl[q_{I^{2h}}(0)-q_{I^{-2h}}(0)\bigr].
\end{equation}

The derivatives of $q_I$ are differences of derivatives of the Gaussian density evaluated at the
shifted endpoints of $I$. We use the bound 
\[
 |\phi^{(k-1)}(v)|
 \le C_k(1+|v|^{k-1})\phi(v).
\]
Suppose first that $v\ge0$. For $r\in[h/2,h]$,
\[
 \frac{\phi(v-r)}{\phi(v)}
 =\exp(vr-r^2/2),
\]
and hence
\[
 \int_{v-h}^{v-h/2}\phi(u)\,\mathrm du
 \ge \frac h2\,\phi(v)\exp(hv/2-h^2/2).
\]
The case $v<0$ follows by symmetry. Since $h\le1/3$,
\[
 h^{k-1}(1+|v|^{k-1})\le C_k e^{h|v|/2},
\]
so
\[
 h^k|\phi^{(k-1)}(v)|
 \le C_k\int_{v-h}^{v+h}\phi(u)\,\mathrm du.
\]

If an endpoint is shifted by at most $h$, the resulting interval of radius $h$ lies within the
radius-$2h$ neighborhood of the original endpoint. Summing the contributions from the two endpoints
therefore gives \eqref{eq:scalar-strip}. The same conclusion holds when $I^{-2h}$ is empty.

\medskip
\emph{Step 2: Summing the coordinate contributions.}
For the coordinate intervals $I_j$, put
\[
 B_j=q_{I_j^h}(0),\qquad
 D_j=q_{I_j^{4h}}(0)-q_{I_j^{-2h}}(0).
\]
Since $|z_j|\le h$, an undifferentiated coordinate factor is at most $B_j$. If coordinate $j$
is differentiated $k\ge1$ times, Step 1 bounds the corresponding factor by
$C_kh^{-k}D_j$. In particular, all $B_j$ are strictly positive.

Consider an ordered derivative $\partial_{j_1\cdots j_m}$, and let
$R\subseteq\{1,\ldots,p\}$ be the set of distinct coordinates among
$j_1,\ldots,j_m$. If $|R|=r$, the product of the coordinatewise bounds is at most
\[
 C_mh^{-m}
 \left(\prod_{j=1}^p B_j\right)
 \prod_{j\in R}\frac{D_j}{B_j}.
\]
For a fixed set $R$, the number of possible multiplicities and orderings
of the $m$ derivatives is bounded by a constant depending only on $m$.
Summing first over these possibilities and then over all sets $R$
bounds the left-hand side of \eqref{eq:envelope} by 
\begin{equation}\label{eq:elementary-sum}
 C_mh^{-m}\left(\prod_{j=1}^pB_j\right)
 \sum_{r=1}^m e_r(D_1/B_1,\ldots,D_p/B_p),
\end{equation}
where $e_r$ denotes the $r$th elementary symmetric polynomial, with $e_r=0$ for $r>p$.

The supremum in \eqref{eq:envelope} is taken separately for each derivative entry, so the
maximizing shift need not be the same from one entry to another.

\medskip
\emph{Step 3: Comparing the products with the boundary-layer probability.}
Set $J_j=I_j^h$, $H=3h=1/\sqrt L$, and
\[
 U_j=q_{J_j^H}(0),\qquad V_j=q_{J_j^{-H}}(0).
\]
Then $B_j=q_{J_j}(0)$ and $D_j=U_j-V_j$.
The remaining estimate is
\begin{equation}\label{eq:product-shell}
 \left(\prod_jB_j\right)\sum_{r=1}^m e_r(D_j/B_j)
 \le C_m\left(\prod_jU_j-\prod_jV_j\right).
\end{equation}
For an interval containing a large central segment, the inward endpoint strips can have much more
mass than the outward strips, so expansion alone does not give a uniform bound for the removed mass.
We treat such intervals separately, using the small total mass of their endpoint strips.
Let $T=4\sqrt L$, and let $F$ consist of the coordinates for which $J_j$ contains $[-T,T]$; write $N$
for its complement.
The choices $H=L^{-1/2}$ and $T=4\sqrt L$ serve two purposes: $TH=4$ bounds density ratios at endpoints
within distance $T$ of the origin, while the tail strips at distance at least $T-H$ have small total
mass over all $p$ coordinates.

For $j\in N$, expansion controls the entire strip probability:
\begin{equation}\label{eq:near}
 D_j\le C(U_j-B_j).
\end{equation}
To see this for an interval on one side of zero, reflect if necessary and write its nearest endpoint
as $a\ge0$. Monotonicity gives
\[
 B_j-V_j\le2\int_a^{a+H}\phi(u)\dd
 \le2\int_{a-H}^a\phi(u)\dd\le2(U_j-B_j).
\]
When $J_j^{-H}$ is nonempty this bounds the two removed strips; when it is empty the interval has
length at most $2H$. The middle inequality holds pointwise after pairing $a+v$ with $a-v$. If the interval contains zero, let $u$ be the distance from zero to its nearer finite endpoint. Since $j\in N$, we have $u\le T$. The removed mass is at most
$2H\phi((u-H)_+)$, and the added strip at that endpoint has mass at least $H\phi(u+H)$. Their ratio is at most $2e^{2uH+H^2/2}\le C$, since $uH\le TH=4$ and $H\le1$. This also handles the case $J_j^{-H}=\varnothing$, proving \eqref{eq:near}.

For $j\in F$, the strip probabilities are instead small in aggregate:
\[
 B_j\ge1-2\bar\Phi(T)\ge\tfrac12,
 \quad D_j\le4H\phi(T-H),
 \quad Q_F=\sum_{j\in F}D_j/B_j\le\tfrac14.
\]
For the last bound, use $T-H\ge3\sqrt L$ and $p=e^{L-1}$ to obtain
$Q_F\le8pL^{-1/2}\phi(3\sqrt L)<1/4$ for $L\ge1$.

For a coordinate set $E$, write $B_E=\prod_{j\in E}B_j$, and similarly for $U_E,V_E$; empty products
equal one.
Choose an absolute $\kappa\in(0,1]$ such that $\kappa D_j\le U_j-B_j$ for $j\in N$, using
\eqref{eq:near}.
The generating product has nonnegative coefficients, so
\begin{align*}
 \kappa^m\left(\prod_jB_j\right)\sum_{r=1}^m e_r(D_j/B_j)
 &\le\prod_j(B_j+\kappa D_j)-\prod_jB_j\\
 &\le B_F\bigl(e^{\kappa Q_F}U_N-B_N\bigr).
\end{align*}
Write
\[
 e^{\kappa Q_F}U_N-B_N
 =e^{\kappa Q_F}(U_N-B_N)+(e^{\kappa Q_F}-1)B_N.
\]
Since $Q_F\le1/4$, we have $e^{\kappa Q_F}\le C$ and
$e^{\kappa Q_F}-1\le C Q_F$. Absorbing $\kappa^{-m}$ gives
\begin{equation}\label{eq:split-products}
 \left(\prod_jB_j\right)\sum_{r=1}^m e_r(D_j/B_j)
 \le C_m[B_F(U_N-B_N)+B_NB_FQ_F].
\end{equation}
For $j\in F$, we have $U_j/B_j\le1+Q_F$, and hence
\begin{align*}
 U_F-V_F
 &=U_F\left[1-\prod_{j\in F}\left(1-\frac{D_j}{U_j}\right)\right]\\
 &\ge B_F\left[1-\exp\left(-\frac{Q_F}{1+Q_F}\right)\right]
 \ge cB_FQ_F.
\end{align*}
The right side of \eqref{eq:split-products} is therefore at most
\[
 C_m[U_F(U_N-B_N)+B_N(U_F-V_F)]
 \le C_m(U_FU_N-V_FV_N).
\]
This proves \eqref{eq:product-shell}. Together with \eqref{eq:elementary-sum}, it gives
\eqref{eq:envelope} for covariance $I_p$, and scaling gives the result for $\sigma^2I_p$.

\medskip
\emph{Step 4: General covariance matrices.}
Let $\Gamma\succeq\sigma^2I_p$. Write
\[
 G_\Gamma=G_{\sigma^2I_p}+H_0,
 \qquad
 H_0\sim N(0,\Gamma-\sigma^2I_p),
\]
with the two Gaussian vectors independent. The covariance of $H_0$ may be singular. Conditioning on
$H_0$ gives
\[
 g_{A,\Gamma}(x)
 =\E\, g_{A,\sigma^2I_p}(x+H_0).
\]
Hence, for every derivative index $(j_1,\ldots,j_m)$,
\[
 \sup_{\norminf z\le h}
 |\partial_{j_1\cdots j_m}g_{A,\Gamma}(x+z)|
 \le
 \E\sup_{\norminf z\le h}
 |\partial_{j_1\cdots j_m}g_{A,\sigma^2I_p}(x+H_0+z)|.
\]
Summing over the derivative indices and applying the bound already proved for covariance
$\sigma^2I_p$, with $y=x+H_0$, yields
\[
 \sum_{j_1,\ldots,j_m=1}^p
 \sup_{\norminf z\le h}
 |\partial_{j_1\cdots j_m}g_{A,\Gamma}(x+z)|
 \le
 C_mh^{-m}\E\!\left[
 g_{A^{4h},\sigma^2I_p}(x+H_0)
 -g_{A^{-2h},\sigma^2I_p}(x+H_0)
 \right].
\]
The right-hand side equals
\[
 C_mh^{-m}
 \bigl[g_{A^{4h},\Gamma}(x)-g_{A^{-2h},\Gamma}(x)\bigr],
\]
which proves \eqref{eq:envelope}. Differentiation under the expectation is justified by the bounded
derivatives of the Gaussian convolution.
\end{proof}

\section{Gaussian comparison}\label{sec:comparison}

Under the assumptions of \zcref{thm:main}, set $\xi_i=X_i/\sqrt n$ and $\delta=B/\sqrt n$, and write
\begin{equation}\label{eq:cutoff}
 r=\delta L^{3/2},\qquad t=144\delta^2L/b^2.
\end{equation}
Then $b\sqrt t/(3\sqrt L)=4\delta$, the largest shift needed below.
For $\lambda\in[0,1]$, set
\[
 a_\lambda=1+\lambda,\qquad
 q_\lambda=\frac{1-\lambda}{(1+\lambda)^2},\qquad
 B_{i,\lambda}\sim\Ber(q_\lambda),\qquad
 Y_{i,\lambda}=a_\lambda B_{i,\lambda}\xi_i,\qquad
 V_\lambda=\sum_iY_{i,\lambda}.
\]
The Bernoulli variables are mutually independent and independent of all $\xi_i$.
Only their laws are needed; no coupling across different $\lambda$ is imposed.
With $M_i=\E\xi_i\xi_i^T$ and $T_i=\E\xi_i^{\otimes3}$,
\begin{equation}\label{eq:thinning-moments}
 \E Y_{i,\lambda}=0,\quad
 \E Y_{i,\lambda}^{\otimes2}=(1-\lambda)M_i,\quad
 \E Y_{i,\lambda}^{\otimes3}=(1-\lambda^2)T_i,\quad
 \norminf{Y_{i,\lambda}}\le2\delta.
\end{equation}
For $A\in\cR_p$, define
\begin{equation}\label{eq:path}
 P_A(\lambda)=\Pp\{V_\lambda+G_{s\Sigma}\in A\},\qquad s=\lambda+t.
\end{equation}
The total covariance is $(1+t)\Sigma$, and
\[
 P_A(0)=\Pp(W+G_{t\Sigma}\in A),\qquad
 P_A(1)=\Pp(G_{(1+t)\Sigma}\in A).
\]

\begin{lemma}\label{lem:path}
Uniformly over $A\in\cR_p$ and $0\le\lambda\le1$,
\begin{equation}\label{eq:path-bound}
 |P_A'(\lambda)|\le C\left[
 b^{-2}\delta L^{3/2}\frac{\lambda}{\lambda+t}
 +b^{-3}\delta^2L^2(\lambda+t)^{-3/2}\right].
\end{equation}
Endpoint derivatives are one-sided.
\end{lemma}
\begin{proof}
Fix $\lambda\in(0,1)$ and write $Y_i=Y_{i,\lambda}$. Set
\[
 V_{-i}=\sum_{k\ne i}Y_k,\qquad g=g_{A,s\Sigma},\qquad s=\lambda+t.
\]
Since $s\ge t>0$, the Gaussian convolution has bounded derivatives of every fixed order.
Together with the boundedness of the summands and the smoothness of $a_\lambda$ and
$q_\lambda$, this justifies the differentiations below. The endpoint derivatives follow by the
same calculation with one-sided derivatives.

We first differentiate the law of a single summand, holding the test function fixed. For a smooth
function $f$,
\[
 \E f(x+Y_i)
 =(1-q_\lambda)f(x)+q_\lambda\E f(x+a_\lambda\xi_i).
\]
Since $a_\lambda'=1$,
\[
 \frac{\mathrm d}{\mathrm d\lambda}\E f(x+Y_i)
 =
 q_\lambda'\E\bigl[f(x+a_\lambda\xi_i)-f(x)\bigr]
 +q_\lambda\E Df(x+a_\lambda\xi_i)[\xi_i].
\]
Expand the first term through order three and the second through order two, with integral
remainders. Centering removes the first-order contribution. Moreover,
\[
 q_\lambda a_\lambda^2=1-\lambda,
 \qquad
 q_\lambda a_\lambda^3=1-\lambda^2,
\]
so the coefficients of the second- and third-order terms are
\[
 \frac12\frac{\mathrm d}{\mathrm d\lambda}(q_\lambda a_\lambda^2)
 =-\frac12,
 \qquad
 \frac16\frac{\mathrm d}{\mathrm d\lambda}(q_\lambda a_\lambda^3)
 =-\frac{\lambda}{3}.
\]
Writing
\[
 M_i=\E\xi_i\xi_i^T,\qquad T_i=\E\xi_i^{\otimes3},
\]
we obtain
\begin{align}\label{eq:identity}
 \frac{\mathrm d}{\mathrm d\lambda}\E f(x+Y_i)
 ={}&-\tfrac12 M_i:D^2f(x)-\tfrac\lambda3 T_i:D^3f(x)\notag\\
 &+\E\int_0^1 K_\lambda(u)
 D^4f(x+ua_\lambda\xi_i)[\xi_i,\xi_i,\xi_i,\xi_i]\dd,
\end{align}
where 
\[
 M:D^2f(x)=\sum_{j,k=1}^p M_{jk}\partial_{jk}f(x),
 \qquad
 T:D^3f(x)=\sum_{j,k,\ell=1}^p T_{jk\ell}\partial_{jk\ell}f(x),
\] and
$D^4f(y)[z,z,z,z]$ denotes the fourth derivative of $f$ at $y$ evaluated four times in
the direction $z$. The two integral remainders combine to give
\begin{align*}
 K_\lambda(u)
 &=\frac{q_\lambda'a_\lambda^4}{6}(1-u)^3
   +\frac{q_\lambda a_\lambda^3}{2}(1-u)^2\\
 &=\frac{1+\lambda}{6}(1-u)^2\bigl[(3-\lambda)u-2\lambda\bigr].
\end{align*}
In particular,
\[
 |K_\lambda(u)|\le1,\qquad 0\le\lambda,u\le1.
\]

We now combine the variation of the summand laws with the variation of the Gaussian component.
Applying the product rule to the independent summand laws, apply \eqref{eq:identity} with $f=g$, average over $V_{-i}$, and sum over $i$. The remaining
$\lambda$-dependence of $g$ comes from $s=\lambda+t$, and the Gaussian heat equation gives
\[
 \partial_\lambda g=\frac12\Sigma:D^2g.
\]
Since $\sum_iM_i=\Sigma$, the quadratic terms combine to
\[
 \frac12\sum_iM_i:
 \E\bigl[D^2g(V_{-i}+Y_i)-D^2g(V_{-i})\bigr].
\]
Expanding each Hessian to first order around $V_{-i}$, the linear term vanishes because
$Y_i$ is centered and independent of $V_{-i}$. Hence
\begin{align}\label{eq:derivative-expansion}
 P_A'(\lambda)={}&-\frac\lambda3\sum_i T_i:\E D^3g(V_{-i})\notag\\
 &+\frac12\sum_i\sum_{j,k,\ell,q}(M_i)_{jk}
 \E\left[
 Y_{i,\ell}Y_{i,q}\int_0^1(1-u)
 \partial_{jk\ell q}g(V_{-i}+uY_i)\dd
 \right]\notag\\
 &+\sum_i\sum_{j,k,\ell,q}
 \E\left[
 \xi_{i,j}\xi_{i,k}\xi_{i,\ell}\xi_{i,q}
 \int_0^1K_\lambda(u)
 \partial_{jk\ell q}g(V_{-i}+ua_\lambda\xi_i)\dd
 \right].
\end{align}

We next bound the coefficients of these derivatives. For every pair of coordinates $j,k$,
Cauchy--Schwarz and the unit diagonal of $\Sigma$ give
\[
 \sum_i\E|\xi_{i,j}\xi_{i,k}|
 \le
 \left(
 \sum_i\E\xi_{i,j}^2
 \sum_i\E\xi_{i,k}^2
 \right)^{1/2}
 =1.
\]
Since $\norminf{\xi_i}\le\delta$, it follows that, for every fixed tuple of coordinate indices,
\begin{align}\label{eq:moment-weights}
 \sum_i|(T_i)_{jk\ell}|
 &\le\delta,\notag\\
 \sum_i\E|\xi_{i,j}\xi_{i,k}\xi_{i,\ell}\xi_{i,q}|
 &\le\delta^2.
\end{align}
Also $\sum_i|(M_i)_{jk}|\le1$, while
\[
 \E|Y_{i,\ell}Y_{i,q}|
 =q_\lambda a_\lambda^2\E|\xi_{i,\ell}\xi_{i,q}|
 =(1-\lambda)\E|\xi_{i,\ell}\xi_{i,q}|
 \le\delta^2.
\]
Thus
\begin{equation}\label{eq:covariance-weights}
 \sum_i|(M_i)_{jk}|\E|Y_{i,\ell}Y_{i,q}|\le\delta^2.
\end{equation}

The derivatives in \eqref{eq:derivative-expansion} are evaluated at different omitted-summand
sums $V_{-i}$. To put them under a common distribution, for each $i$ introduce an independent
copy $Y_i^0$ of $Y_i$, independent of all variables appearing in the $i$th remainder, and set
\[
 \widehat V_i=V_{-i}+Y_i^0.
\]
Then
\[
 \widehat V_i\stackrel{d}{=}V_\lambda.
\]
Moreover,
\[
 \norminf{Y_i^0}\le2\delta,\qquad
 \norminf{Y_i}\le2\delta,\qquad
 \norminf{a_\lambda\xi_i}\le2\delta.
\]
Consequently, every derivative argument in \eqref{eq:derivative-expansion} lies within
$4\delta$ in sup norm of $\widehat V_i$. For $m=3,4$, define
\[
 H_{j_1\cdots j_m}(x)
 =
 \sup_{\norminf z\le4\delta}
 |\partial_{j_1\cdots j_m}g(x+z)|.
\]
The variables supplying the coefficients in each term of
\eqref{eq:derivative-expansion} are independent of $\widehat V_i$. For example,
\[
 \E\!\left[
 |\xi_{i,j}\xi_{i,k}\xi_{i,\ell}\xi_{i,q}|
 H_{jk\ell q}(\widehat V_i)
 \right]
 =
 \E|\xi_{i,j}\xi_{i,k}\xi_{i,\ell}\xi_{i,q}|\,
 \E H_{jk\ell q}(V_\lambda).
\]
The covariance term factors in the same way, while the coefficients $T_i$ in the cubic term
are deterministic. Using \eqref{eq:moment-weights} and \eqref{eq:covariance-weights} therefore
gives
\begin{equation}\label{eq:aggregate}
 |P_A'(\lambda)|
 \le
 C\lambda\delta\,
 \E\sum_{j,k,\ell}H_{jk\ell}(V_\lambda)
 +
 C\delta^2\,
 \E\sum_{j,k,\ell,q}H_{jk\ell q}(V_\lambda).
\end{equation}

Finally, the coordinate variances of $V_\lambda$ are $1-\lambda$, while
$s=\lambda+t$, so
\[
 \max\{1-\lambda,s\}\ge\frac12.
\]
The summands of $V_\lambda$ are bounded by $2\delta$, and \eqref{eq:cutoff} gives
\[
 \frac{b\sqrt s}{3\sqrt L}\ge4\delta.
\]
Thus \zcref{cor:averaged} applies with $\eta=2\delta$, first with $m=3$ and then with
$m=4$. Substituting the resulting bounds into \eqref{eq:aggregate} yields
\[
 |P_A'(\lambda)|
 \le C\left[
 b^{-2}\delta L^{3/2}\frac{\lambda}{\lambda+t}
 +
 b^{-3}\delta^2L^2(\lambda+t)^{-3/2}
 \right],
\]
which is \eqref{eq:path-bound}.
\end{proof}

Since
\[
 \int_0^1\frac{\lambda}{\lambda+t}\,\mathrm d\lambda\le1,
 \qquad \int_0^1(\lambda+t)^{-3/2}\,\mathrm d\lambda\le2t^{-1/2},
\]
integration of \zcref{lem:path} gives
\begin{equation}\label{eq:integrated}
 \sup_{A\in\cR_p}|P_A(1)-P_A(0)|
 \le Cb^{-2}r+Cb^{-3}\frac{\delta^2L^2}{\sqrt t}
 \le Cb^{-2}r.
\end{equation}

\begin{proof}[Proof of \zcref{thm:main}]
It remains to remove the Gaussian smoothing at the two endpoints of the interpolation.
For deterministic $z\in\R^p$, set $a=\norminf z$. Then
\[
 A^{-a}\subseteq A\cap(A-z),
 \qquad
 A\cup(A-z)\subseteq A^a.
\]
Hence a change in membership under translation by $z$ can occur only in the boundary layer
$A^a\setminus A^{-a}$.

Applying \zcref{lem:concentration} to $W$ therefore gives
\[
 |\Pp(W+z\in A)-\Pp(W\in A)|
 \le C(a+\delta)\sqrt L.
\]
Similarly, \zcref{cor:gauss-shell} gives, for $G_\Sigma\sim N(0,\Sigma)$,
\[
 |\Pp(G_\Sigma+z\in A)-\Pp(G_\Sigma\in A)|
 \le Ca\sqrt L.
\]

Now let $G_{t\Sigma}\sim N(0,t\Sigma)$ be independent. Since each of its coordinates has variance
$t$,
\[
 \E\norminf{G_{t\Sigma}}\le C\sqrt{tL}.
\]
Averaging the preceding inequalities over $G_{t\Sigma}$ yields
\[
 \rho_{\cR}(W,W+G_{t\Sigma})
 \le C(\sqrt t\,L+\delta\sqrt L),
\]
and
\[
 \rho_{\cR}(G_\Sigma,G_{(1+t)\Sigma})
 \le C\sqrt t\,L,
\]
where we used
$G_\Sigma+G_{t\Sigma}\stackrel{d}{=}G_{(1+t)\Sigma}$.

By construction,
\[
 P_A(0)=\Pp(W+G_{t\Sigma}\in A),
 \qquad
 P_A(1)=\Pp(G_{(1+t)\Sigma}\in A).
\]
Thus \eqref{eq:integrated} and the triangle inequality give
\[
 \rho_{\cR}(W,Z)
 \le
 \sup_{A\in\cR_p}|P_A(1)-P_A(0)|
 +C(\sqrt t\,L+\delta\sqrt L)
 \le Cb^{-2}r.
\]
Indeed, the choice $t\asymp\delta^2L/b^2$ gives
\[
 \sqrt t\,L+\delta\sqrt L\le Cb^{-1}r\le Cb^{-2}r.
\]
Finally, $\rho_{\cR}(W,Z)\le1$ and $\delta=B/\sqrt n$, so, after increasing the absolute
constant,
\[
 \rho_{\cR}(W,Z)
 \le
 C\min\left\{1,\frac{B\log^{3/2}(ep)}{b^2\sqrt n}\right\}.
\]
\end{proof}

\bibliographystyle{plain}
\bibliography{sharp_rectangle_clt}
\end{document}